\documentclass[12pt]{amsart}
\usepackage{a4wide}
\usepackage{amsmath,amsfonts,amssymb}
\usepackage{amsthm}
\usepackage{ifthen}
\usepackage{longtable}
\usepackage{array}
\usepackage{url}
\usepackage[utf8]{inputenc}
\usepackage{enumerate}
\usepackage{multirow}
\allowdisplaybreaks

\usepackage{color}

\newcommand{\opi}{\overline{\pi}}

\newcommand{\bbF}{\mathbb{F}}
\newcommand{\bbFq}{\mathbb{F}_q}

\newcommand{\bbFqn}{\mathbb{F}_{q^n}}
\newcommand{\calH}{\mathcal{H}}
\newcommand{\calT}{\mathcal{T}}
\newcommand{\ord}{\mathrm{ord}}
\newcommand{\bbL}{\mathbb{L}}
\newcommand{\calP}{\mathcal{P}}
\newcommand{\calI}{\mathcal{I}}
\newcommand{\calS}{\mathcal{S}}
\newcommand{\calO}{\mathcal{O}}
\newcommand{\calM}{\mathcal{M}}
\newcommand{\calX}{\mathcal{X}}
\newcommand{\bfx}{\mathbf{x}}

\newcommand{\bbZ}{\mathbb{Z}}

\newcommand{\bbQ}{\mathbb{Q}}

\newcommand{\bbR}{\mathbb{R}}

\newcommand{\bbC}{\mathbb{C}}

\DeclareMathOperator{\Norm}{N}

\newtheorem*{theorem*}{Theorem}
\newtheorem{theorem}{Theorem}
\newtheorem{lemma}{Lemma}

\theoremstyle{remark}

\begin{document}
\title[Perfect squares representing the number of rational points on elliptic curves]{Perfect squares representing the number of rational points on elliptic curves over finite field extensions}

\author[K. C. Chim]{Kwok Chi Chim}
\address{K. C. Chim \newline
 Institute of Analysis and Number Theory, Graz University of Technology,
Kopernikusgasse 24/II,
A-8010 Graz, Austria.}
\email{chim\char'100math.tugraz.at}

\author[F. Luca]{Florian Luca}
\address{F. Luca \newline
School of Mathematics, Wits University, Johannesburg, South Africa \newline
And\newline
Research Group in Algebraic Structures and Applications, King Abdulaziz University, Jeddah, Saudi Arabia \newline
And\newline
Max-Planck-Institute for Mathematics, Vivatsgasse 7, 53111 Bonn, Germany. \newline
And\newline
Centro de Ciencias Matem\'aticas, UNAM, Morelia, Mexico }
\email{florian.luca\char'100wits.ac.za}

\begin{abstract}
Let $q$ be a perfect power of a prime number $p$ and $E({\mathbb F}_q)$ be an elliptic curve over ${\mathbb F}_q$ given by the equation $y^2=x^3+Ax+B$.  For a positive integer $n$ 
we denote by
$ \# E(\bbF_{q^n})$  the number of rational points on $E$ (including infinity) over the extension $\bbF_{q^n}$. 
Under a mild technical condition, we show that the sequence  $\lbrace \# E(\bbF_{q^n}) \rbrace_{n>0}$ 
contains at most $10^{200}$ perfect squares.   If the mild condition is not satisfied, then $\#E({\mathbb F}_{q^n})$ is a perfect square for infinitely many $n$ including all the multiples of $24$. 
Our proof uses a quantitative version of the Subspace Theorem. We also find all the perfect squares for all such sequences in the range $q < 50$ and $n\leq 1000$. 
\end{abstract}

\maketitle

\section{Introduction}\label{Sec:Intro}

Let $q$ be a prime power. We denote by $E/\bbFq$ the elliptic curve $y^2=x^3 +Ax+ B$  over the finite field $\bbFq$ of $q$ elements. The condition that $E$ is an elliptic curve
is equivalent to $4A^3+27B^2\ne 0$ in ${\mathbb F}_q$. It is known that if $p>3$, then any elliptic curve over ${\mathbb F}_p$ admits such an equation.
We consider the set  $E(\bbFq)$ of $\bbFq$-rational points on $E$ (including the point at infinity) over the finite field $\bbFq$. This set forms an abelian group with the Mordell-Weil addition law with the point at infinity as the identity (zero) element. 
We denote by $ \# E(\bbFq)$ the cardinality of this group. 
It is well-known by a theorem of Hasse that 
\begin{align*}
\left|  q + 1 -\# E(\bbFq) \right| \leq 2\sqrt{q}.
\end{align*}
We set: 
\begin{align*}
a=q+1-\# E(\bbFq).
\end{align*}
The number $a$ is called the trace of the Frobenius. For a positive integer $n$ we look at the set of points of $E$ including the point at infinity over the field extension $\bbFqn$ and denote this set by $E(\bbFqn)$. 
This is again an abelian group having $E({\mathbb F}_q)$ as a subgroup. 
We denote by $ \# E(\bbFqn)$ the cardinality of this group. 
Again by Hasse's theorem, we have
\begin{align*}
\left|  q^n + 1 -\# E(\bbFqn) \right| \leq 2\sqrt{q^n}.
\end{align*}

There is a formula which allows us to compute $\#E({\mathbb F}_{q^n})$ from knowledge of $q,~\#E({\mathbb F}_q)$ and $n$ only. Namely, let  
$\alpha, \beta \in \bbC$ be the roots of the quadratic polynomial $x^2 -ax +q$. Then $\alpha$ and $\beta$ are complex conjugates satisfying $\alpha+\beta=a$ and  
$|\alpha| = |\beta| = \sqrt{q}$. Furthermore, for every $n \geq 1$, the following formula holds
\begin{equation*}
\# E(\bbFqn) = 
q^n + 1 - a_n, 
\end{equation*}
 where $a_n=\alpha^n + \beta^n$.

The group $E(\bbFqn)$ has drawn much attention in literature. 
For instance, Luca and Shparlinski \cite{LucaShpar} consider $l(q^n)$, the exponent of $E(\bbFqn)$; that is, the largest possible order of points $P \in  E(\bbFqn)$, and show that $l(q^n)\ge q^{n(1-\varepsilon)}$
holds for all $\varepsilon>0$ as $n>n_{\varepsilon}$ with some (ineffective) value of $n_{\varepsilon}$. In this paper, we look at perfect squares in the sequence $\lbrace \# E(\bbFqn) \rbrace_{n>0}$. We show that under mild conditions, there are at most $10^{200}$ values of $n$
such that $\#E({\mathbb F}_{q^n})$ is a perfect square. These potential values are not effectively computable. If the mild conditions are not satisfied, then $\#E(F_{q^n})$ is a perfect square for infinitely many 
$n$ including all the multiples of $24$. Below is our concrete result.

\begin{theorem} \label{Thm_Main}
Let $E$ be an elliptic curve over ${\mathbb F}_q$ given by the equation $y^2=x^3+Ax+B$ and let $\alpha,~\beta$ be the two roots of $x^2-ax+q=0$, where $a=q+1-\#E({\mathbb F}_q)$.
\begin{enumerate}[(a).]
\item If ${\alpha}/{\beta}$ is not a root of unity, then there exists at most $5.6\times 10^{194}$ perfect squares in the sequence $\lbrace \# E(\bbFqn) \rbrace_{n>0}$.
\item If ${\alpha}/{\beta} = \zeta = e^{\frac{2\pi ki}{m}}$ is  a root of unity with $\gcd(k,m)=1$, then  $m=1, 2, 3, 4, 6$. Furthermore,  there exist infinitely many perfect squares  in the sequence $\lbrace \# E(\bbFqn) \rbrace_{n>0}$. 
In particular, $ \# E(\bbFqn) $ is a perfect square if $n \equiv 0 \;\mathrm{mod}\;m$. The following table shows the pattern.

\begin{center}
\begin{tabular}{| c | c | c | c | c | c |} 
\hline 
$m$ & $q$  & $a$  & $n$ & $a_n$ & $(u(q, a,  n))^2$ \\[1mm]
\hline 
\multirow{2}{6mm}{$\;\,1$} & \multirow{2}{22mm}{$\;\;p^{2v}, v\in \bbZ$}  & $2p^v$  & $ \forall n \in \bbZ$ & $2p^{vn}$ & $(p^{vn}-1)^2$   \\
  &    & $-2p^v$  & $ \forall n \in \bbZ$ & $2(-p^v)^n$ & $((-p^{v})^n - 1)^2$   \\
  \hline  
\multirow{2}{6mm}{$\;\,2$} & \multirow{2}{45mm}{$ \forall q$ except $q=p^{2v}$ where $p \equiv 1 \; \mathrm{mod}\;4, v\in \bbZ$}  & \multirow{2}{10mm}{$\quad 0$}  & $n\equiv 2 \; \mathrm{mod}\; 4$ & $-2q^{n/2}$ & $( q^{n/2}+1)^2$  \\
 &    &   & $n\equiv 0 \; \mathrm{mod}\;4$ & $2q^{n/2}$ & $( q^{n/2}-1)^2$  \\
 \hline 
\multirow{4}{6mm}{$\;\,3$} & \multirow{4}{32mm}{$ p^{2v}, v\in \bbZ$  except $p \equiv 1 \; \mathrm{mod}\;3$}  & \multirow{2}{10mm}{$\quad p^v$}  & $n\equiv 0 \; \mathrm{mod}\;6$ & $2p^{vn}$ & $(q^{n/2}+1)^2$   \\
 &    &    & $n\equiv 3 \; \mathrm{mod}\;6$ & $2p^{vn}$ & $(q^{n/2}-1)^2$   \\ \cline{3-6}
  &    & \multirow{2}{10mm}{$\;\, -p^v$}  & $n\equiv 0 \; \mathrm{mod}\;6$ & $2p^{vn}$ & $(q^{n/2}-1)^2$ \\
  &  &  & $n\equiv 3 \; \mathrm{mod}\;6$ & $-2p^{vn}$ & $(q^{n/2}-1)^2$ \\ \hline 
\multirow{4}{6mm}{$\;\,4$} & \multirow{2}{22mm}{$ 2^v$, $v$ is odd}  & \multirow{2}{10mm}{$\;2^{\frac{v+1}{2}} $}  & $n \equiv 0 \; \mathrm{mod}\;8$ & $2^{\frac{nv+2}{2}}$ & $(q^{n/2}-1)^2$   \\
  &   &   & $n \equiv 4 \; \mathrm{mod}\;8$ & $-2^{\frac{nv+2}{2}}$ & $(q^{n/2}+1)^2$  \\ \cline{2-6}
  & \multirow{2}{22mm}{$2^v$, $v$ is odd}  & \multirow{2}{10mm}{$-2^{\frac{v+1}{2}} $}  & $n \equiv 0 \; \mathrm{mod}\;8$ & $2^{\frac{nv+2}{2}}$ & $(q^{n/2}-1)^2$   \\
 &    &   & $n \equiv 4 \; \mathrm{mod}\;8$ & $-2^{\frac{nv+2}{2}}$ & $(q^{n/2}+1)^2$   \\ \hline
\multirow{2}{6mm}{$\;\,6$} & \multirow{2}{22mm}{$3^v$, $v$ is odd}  & \multirow{2}{10mm}{$\pm 3^{\frac{v+1}{2}} $}  & $n \equiv 0 \; \mathrm{mod}\;12$ & $2(3^{\frac{nv }{2}} )$ & $(q^{n/2}-1)^2$  \\
  &    &     & $n \equiv 6 \; \mathrm{mod}\;12$ & $-2(3^{\frac{nv }{2}} )$ & $(q^{n/2}+1)^2$   \\ 
\hline  
\end{tabular}
\end{center}
$ $\
\item[(c).] With the assumptions and notations from (b),  the only perfect squares in the sequence $\lbrace \# E(\bbFqn) \rbrace_{n>0}$ of the form  $(u(q, a, n))^2$ with $n\not\equiv 0\pmod m$ are  
\begin{align*}
&  1  = 1^2 =  (u(2, 2, 1))^2 = (u(3, 3, 1))^2, \\
 & 4 = 2^2 =  (u(3, 0, 1))^2, \\
 &  9 = 3^2 =  (u(2, 0, 3))^2  =  (u(8, 0, 1))^2, \\
& 25 = 5^2 = (u(2, -2, 5))^2=(u(32,8,1))^2.
\end{align*}
\end{enumerate}
\end{theorem}
 
We ran a computation that looked at all elliptic curves over $\bbFqn$ for $q < 50$ and their extensions to $F_{q^{n}}$ for $1 \leq n \leq 10^3$. 
We extracted all the terms which are perfect squares using SAGE. Let $q$, $a$, $a_n$, $\alpha$ and $\beta$ be defined as before. The following is the numerical result.

\begin{theorem} \label{Thm_Sage}
For the sequence $\lbrace \# E(\bbFqn) \rbrace_{n>0}$ with ${\alpha}/{\beta}$ not a root of unity, 
the terms within $q < 50$ and $1 \leq n \leq 10^3$ which are perfect square, that is, for which $(u(q, a, n))^2 = \# E(\bbFqn)$,  are the following
\begin{align*}
4 & = 2^2 = (u(2, -1, 1))^2 = (u(2, -1, 3))^2 = (u(4, 1, 1))^2   =  (u(5, 2, 1))^2   \\
& \phantom{ = 2^2\; }= (u(7, 4, 1))^2  = (u(8, 5, 1))^2, \\
9 & = 3^2 =   (u(5, -3, 1))^2   = (u(7, -1, 1))^2   = (u(9, 1, 1))^2  =  (u(11, 3, 1))^2  \\
 & \phantom{ = 3^2\; }  =  (u(13, 5, 1))^2 , \\
16 & = 4^2 = (u(2, -1, 4))^2 = (u(2, 1, 4))^2 =  (u(4, -3, 2))^2 =  (u(4, 3, 2))^2  \\
& \phantom{ = 4^2\; }   =  (u(11, -4, 1))^2   =  (u(13, -2, 1))^2 =  (u(16, 1, 1))^2  =  (u(17, 2, 1))^2   \\
& \phantom{ = 4^2\; }   =  (u(19, 4, 1))^2   = (u(23, 8, 1))^2, \\
25 & = 5^2 
=  (u(17, -7, 1))^2   = (u(19, -5, 1))^2   = (u(23, -1, 1))^2 =  (u(25, 1, 1))^2 \\
& \phantom{ = 5^2\; } = (u(27, 3, 1))^2 =   (u(29, 5, 1))^2   = (u(31, 7, 1))^2   = (u(32, 8, 1))^2, \\
36 & = 6^2 = (u(3, 1, 3))^2 =   (u(27, -8, 1))^2  =  (u(29, -6, 1))^2   =   (u(31, -4, 1))^2  \\
 & \phantom{ = 6^2\; } = (u(36, 1, 1))^2  = (u(37, 2, 1))^2   = (u(41, 6, 1))^2   = (u(43, 8, 1))^2      \\
  & \phantom{ = 6^2\; }   =   (u(47, 12, 1))^2,  \\
49 &  =  7^2 =   (u(37, -11, 1))^2   =   (u(41, -7, 1))^2   =  (u(43, -5, 1))^2   =   (u(47, -1, 1))^2,  \\
& \phantom{ = 7^2\; }   =  (u(49, 1, 1))^2,  \\
144 & = 12^2 =  (u(5, 3, 3))^2,  \\
324 & = 18^2 =  (u(7, -4, 3))^2   =   (u(7, -1, 3))^2   = (u(7, 5, 3))^2,  \\
2116 & = 46^2 = (u(2, -1, 11))^2, \\
3025 & = 55^2 = (u(5, 1, 5))^2,  \\
4900 &  =  70^2 =  (u(17, -7, 3))^2,  \\
12100 &  = 110^2 =  (u(23, -1, 3))^2,  \\
24336 &  =  156^2 =  (u(29, -9, 3))^2,  \\
103684 &  = 322^2 =  (u(47, -1, 3))^2  .
\end{align*}

\end{theorem}
We refrain from listing the corresponding elliptic curves for each Frobenius stated above as they can be readily obtained by computer.\\

We appeal to a quantitative version of Subspace Theorem to obtain the explicit upper bound for $n$ for the case when the number of perfect squares is finite in Theorem \ref{Thm_Main}.
In Section \ref{Sec:Prelim}, we shall present several lemmas for the preparation of Theorem \ref{Thm_Main}. The proof of Theorem \ref{Thm_Main} is presented in Section \ref{Sec:ProofThmMain}.

\section{Notations and Preliminary results}\label{Sec:Prelim}

We first present a lemma concerning the criteria on $a$ such that there is corresponding elliptic curve over the finite field $\bbFq$. It is rephrased from \cite[Theorem 4.1]{Waterhouse}.

\begin{lemma} \label{Lemma_Waterhouse}
Let $q = p^b$ be a perfect power of $p$ and let $N = p^b + 1 - a$. 
Then there exists an elliptic curve $E/\bbFq$ such that 
$\#E(\bbF_q) = N $
if and only if 
$|a| \leq 2\sqrt{q}$ and one of the following is satisfied:
\begin{enumerate} [(a).]
\item $\gcd(a, p)=1$;
\item $b$ is even and one of the following is satisfied:
\begin{enumerate} [(i).]
\item $a= \pm 2\sqrt{q}$;
\item $p \not\equiv 1 \;\mathrm{mod}\;3$, and $a = \pm \sqrt{q}$;
\item $p \not\equiv 1 \;\mathrm{mod}\;4$, and $a = 0$;
\end{enumerate}
\item $b$ is odd, and one of the following is satisfied:
\begin{enumerate} [(i).]
\item $p=2$ or $3$, and $a=\pm p^{(b+1)/2}$;
\item $a = 0$.
\end{enumerate}
\end{enumerate}
\end{lemma}

Before we present the next lemma, we give a review on some standard notations of algebraic number theory, Diophantine equations and Diophantine approximations. \\

Let $\bbL$ be an algebraic number field of degree $D$ over $\bbQ$. Denote its ring of integers by $\calO_{\bbL}$ and its collection of places by $\calM_\bbL$. For a fractional ideal $\calI$ of $\bbL$, we denote by $\Norm_{\bbL}(\calI)$ its norm. We note that $\Norm_\bbL(\calI) = \#(\calO_\bbL/\calI)$ if $\calI$ is an ideal of $\calO_\bbL$, and the norm map is extended multiplicatively (using unique factorization) to all the fractional ideals of $\bbL$.\

For a prime ideal $\calP$, we denote by $\ord_\calP(x)$ the order at which it appears in the factorization of the principal ideal $x{\mathcal O}_{\mathbb L}$ generated by $x$ inside $\bbL$.\\

For $\mu \in \calM_\bbL$ and $x \in \bbL$, we define the absolute value $|x|_\mu$ as follows:
\begin{enumerate}[(i).]
\item $|x|_\mu: = |\sigma(x)|^{1/D}$ if $\mu$ corresponds to the embedding $\sigma : \bbL \hookrightarrow \bbR$;
\item $|x|_\mu := |\sigma(x)|^{2/D} = |\overline{\sigma}(x)|^{2/D}$ if $\mu$ corresponds to the pair of complex conjugate embeddings $\sigma, ~\overline{\sigma}: \bbL \hookrightarrow \bbC$;
\item $|x|_\mu := \Norm_\bbL(\calP)^{-\ord_{\calP}(x)/D} $ if $\mu$ corresponds to the the nonzero prime ideal $\calP$ of $\calO_\bbL$;
\end{enumerate}

We say that $\mu$ is real infinite or complex infinite in case (i) or (ii) respectively, and we say that $\mu$ is finite in case (iii). 

We note that these absolute values satisfy the product formula
\begin{equation} \label{prod1}
\prod_{\mu \in \calM_\bbL} |x|_\mu =1, \qquad \forall x \in \bbL^*.
\end{equation}

Let $m \geq 1$ be a positive integer. For $\mu \in \calM_\bbL$ and a nonzero vector $\bfx \in \bbL^m$, we put
\begin{align*}
|\bfx|_\mu &:= \left( \sum_{i=1}^m |x_i|_\mu^{2D} \right)^{1/2D} && \mbox{if $\mu \in \calM_\bbL$ is real infinite,} \\
|\bfx|_\mu &:= \left( \sum_{i=1}^m |x_i|_\mu^D \right)^{1/D} && \mbox{if $\mu \in \calM_\bbL$ is complex infinite,} \\
|\bfx|_\mu &:= \max\left(|x_1|_\mu, \dotso, |x_m|_\mu \right) && \mbox{if $\mu \in \calM_\bbL$ is finite.} 
\end{align*}
We note that for an infinite place $\mu$, $|\bfx|_\mu$ is a power of the standard Euclidean norm $\|\bfx\|$.\\

Next, we define the height of a nonzero vector $\bfx \in \bbL^m\backslash \{{\bf 0}\}$ as follows:
\begin{equation*}
H(\bfx) := H(x_i, \dotso, x_m) = \prod_{\mu \in \calM_\bbL} |\bfx|_\mu.
\end{equation*}

For $x \in \bbL$, we put
\begin{equation*}
\calH(x) := H((1, x)).
\end{equation*}

For a linear form 
\begin{align*}
L(\bfx) = \sum_{i=1}^m a_i x_i \in \bbL(x_1, \dotso, x_m),
\end{align*}
we define $H(L):=H(\mathbf{a})$, where $\mathbf{a} = (a_1, \dotso, a_m)$. We also put $|L|_\mu := |\mathbf{a}|_\mu$ for $\mu \in \calM_\bbL$.\\

We now state the explicit version of the Subspace Theorem by Evertse and Schlickewei \cite[Theorem 3.1]{EvertseS2002} (see also \cite{Evertse1996}).\

\begin{lemma}
 \label{Lemma_Subspace}
Let $\calS$ be a finite set of places of $\bbL$ containing all the infinite ones. Let $\lbrace L_{1,\mu}, \dotso, L_{m, \mu} \rbrace$ for $\mu \in \calS$ be $m > 1$ linearly independent sets of linear forms with coefficients in $\bbL$ such that for some real $H >0$, the inequality
\begin{align*}
H(L_{i,\mu}) \leq H
\end{align*}
holds for all $\mu \in \calS$ and $i=1, \dotso, m$. For a fixed $0<\delta <1$, consider the set $\calX$ of solutions $\bfx \in \bbL^m \setminus \lbrace \mathbf{0} \rbrace $ of the inequality
\begin{align} \label{Subspace_ToSatisfy}
\prod_{\mu \in \calS} \prod_{i=1}^m \frac{|L_{i,\mu}(\bfx)|_{\mu}}{|\bfx|_\mu} < H(\bfx)^{-m-\delta} \prod_{\mu \in \calS} |\det(L_{1,\mu}, \dotso, L_{m,\mu})|_\mu.
\end{align}
Then there exist $t$ proper linear subspaces $\calT_1, \dotso, \calT_t$ of $\bbL^m$, with
\begin{equation*}
t \leq (3m)^{2m \# \calS} 2^{3(m+9)^2} \delta ^{-m \# \calS - m - 4} \log 4D \log \log 4D
\end{equation*}
such that every solution $\bfx \in \calX$ with
\begin{equation*}
H(\bfx) \geq \max \lbrace m^{4m/\delta}, H \rbrace
\end{equation*}
belongs to $\calT_1 \cup \cdots \cup \calT_t$.
\end{lemma}


We next present a lemma used in \cite{LucaShpar} concerning an upper bound on the zero multiplicity of a nondegenerate linear recurrence sequence 
$\lbrace u_n \rbrace_{n\in \bbZ}$ of complex numbers due to van der Poorten and Schlickewei \cite{vdPSch1991}. 

\begin{lemma}  \label{Lemma_vdPSch}
Let $K \geq 1$ be an integer, and let $\alpha_i, \beta_i \in {\mathcal O}_\bbL  \backslash \lbrace 0\rbrace$, $i=1, \dotso, K$, such that $\alpha_i/\alpha_j$ is not a root of unity for any $1\leq i<j\leq K$. Then, the number of solutions $s$ of the equation
\begin{align*}
\sum_{j=1}^K \beta_j \alpha_j^n =0, \qquad n=1, 2, \dotso, 
\end{align*}
satisfies the inequality
\begin{align*}
s\leq (K-1)(4(D+\omega))^{2\omega+1},
\end{align*}
where $\omega$ is the number of prime ideal divisors of $\alpha_1 \cdots \alpha_K$ in ${\mathcal O}_\bbL$.
\end{lemma}

 
It should be noted that there are also results concerning upper bounds on the zero multiplicity of a nondegenerate linear recurrence sequence 
$\lbrace u_n \rbrace_{n\in \bbZ}$ of complex numbers by  Schmidt \cite{Schmidt99} and by Schlickewei and Schmidt \cite{SchlickeweiSchmidt2000}, among others. These bounds are more general but they have a worse dependence in the parameter $K$, which is why we do not use them here.

Finally, we present a technical lemma whose aim is to give a more concise proof for Theorem \ref{Thm_Main}.

\begin{lemma} \label{Lemma_O(k+1)}
Let $d$ be a positive integer, $z \in \bbC$ with $|z|<1$. Then for the complex function 
$f(w)=(1+w)^{1/d}$ we have
\begin{align*}
\left\vert f(w)- \sum_{r=0}^k \binom{1/d}{r} w^r \right\vert
=\left\vert \sum_{r=k+1}^\infty \binom{1/d}{r} w^r \right\vert
\leq \frac{1}{d(k+1)(1-|w|)} |w|^{k+1},
\end{align*}
where we have chosen the branch of $(1+z)^{1/d}$ which is holomorphic on $\bbC \setminus (-\infty, -1]$ and which is equal to the positive $d$-th root of $(1+z)$ for $z\in \bbR$, $z>-1$. \\
In particular, when $d=\frac{1}{2}$,  $q=\alpha \beta$, 
$q \geq 2$ and $n \geq 30$,  
if $w= \frac{1}{q^n} -\frac{\alpha^n}{q^n } - \frac{\beta^n}{q^n} $, we have
 \begin{align*}
\left\vert \sum_{r=1}^\infty \binom{1/2}{r} w^r \right\vert 
 <  0.0003
\qquad \mbox{and} \qquad
\left\vert \sum_{r=2}^\infty \binom{1/2}{r} w^r \right\vert
 <  \frac{4.001}{q^{n} }, 
\end{align*}
for $k=0$ and $k=1$, respectively.
\end{lemma}
\begin{proof}
See for example  \cite[Lemma 2]{FuchsTichy2003} for the proof of the first inequality.\\
We shall prove the remaining inequalities. For $w= \frac{1}{q^n} -\frac{\alpha^n}{q^n } - \frac{\beta^n}{q^n} $, $q=\alpha \beta$, $q \geq 2$ and $n \geq 30$, we have
\begin{align*}
|w| &= \left\vert \frac{1}{q^n} -\frac{\alpha^n}{q^n } - \frac{\beta^n}{q^n}\right\vert 
\leq \left\vert \frac{1}{q^{n}}\right\vert  +   \frac{2}{|q|^{n/2} } 
\leq  \frac{1}{2^{30}}  +   \frac{2}{ 2 ^{15} } < 0.0001,
\end{align*}
and
\begin{align*}
|w|^2 & 
= \left\vert \frac{1}{q^n} -\frac{\alpha^n}{q^n } - \frac{\beta^n}{q^n}\right\vert^2
= \left\vert \frac{1}{q^{2n}} - \frac{2}{q^n } \left( \frac{\alpha^n}{q^n } + \frac{\beta^n}{q^n } \right)
 +   \frac{\alpha^{2n}}{q^{2n} } +  \frac{2\alpha^n \beta^n}{q^{2n} } + \frac{\beta^{2n}}{q^{2n} }  \right\vert \\
& \leq   \frac{1}{q^{2n}} + \frac{4}{q^{3n/2} } 
 +   \frac{4}{q^{n} }  
=  \left( \frac{1}{q^{n}} + \frac{4}{q^{n/2} } 
 +   4 \right) \frac{1}{q^{n} } 
\leq  \left( \frac{1}{2^{30}} + \frac{4}{2^{15} } 
 +   4 \right) \frac{1}{q^{n} }  
<  \frac{4.0002}{q^{n} }. 
\end{align*}
Thus, when $d=\frac{1}{2}$, we get
\begin{align*}
\left\vert \sum_{r=1}^\infty \binom{1/2}{r} w^r \right\vert 
\leq \frac{2|w|}{(1-|w|)}   
< 0.0003
\quad \mbox{and} \quad
\left\vert \sum_{r=2}^\infty \binom{1/2}{r} w^r \right\vert 
\leq \frac{|w|^{2} }{(1-|w|)} 
< \frac{ 4.001 }{q^{n} } 
\end{align*}
for $k=0$ and $k=1$ respectively, which is what we wanted. 
\end{proof}

\section{Proof of Theorem \ref{Thm_Main}(a)}  \label{Sec:ProofThmMain}

We denote by $(u(n))^2$, where $u(n) \in \bbZ$, the term in the sequence $\lbrace \#E(\bbFqn)\rbrace_{n>0}$ which can be expressed as a perfect square. We have
\begin{align} \label{u2}
u^2=(u(n))^2 = q^n + 1 - \alpha^n - \beta^n.
\end{align}
Let $w= \frac{1}{q^n} -\frac{\alpha^n}{q^n } - \frac{\beta^n}{q^n} $. We can rewrite \eqref{u2} as
\begin{equation}  \label{u_def}
\begin{split}
u & =  q^{n/2} \left( 1 + w \right)^{1/2} 
 = q^{n/2} \left( 1 + \sum_{r=1}^k \binom{1/2}{r} w^r 
+ \sum_{r=k+1}^\infty \binom{1/2}{r} w^r \right).
\end{split}
\end{equation}
 
We assume that $n \geq 20$ and $q \geq 2$. With $k=1$, we get
\begin{align*}
 u - q^{n/2} \left( 1 + \frac{1}{2}\left(\frac{1}{q^n} -\frac{\alpha^n}{q^n } - \frac{\beta^n}{q^n} \right)
   \right) =  q^{n/2} \sum_{r=2}^\infty \binom{1/2}{r} w^r.  
\end{align*}

 By Lemma \ref{Lemma_O(k+1)}, we obtain the approximation 
\begin{align*}
& \left\vert u - q^{n/2}   + \frac{1}{2} \left(   \frac{\alpha^n}{q^{n/2} }  \right)  +  \frac{1}{2} \left(  \frac{\beta^n}{q^{n/2}}     \right) \right\vert 
<  \frac{ 4.6 }{q^{n/2} },
\end{align*}
or equivalently
\begin{align} \label{LinearForm}
& \left\vert u q^{n/2} - q^{n}  + \frac{\alpha^n}{2 } + \frac{\beta^n}{2 } 
   \right\vert 
< 4.6. 
\end{align}

We apply Lemma \ref{Lemma_Subspace}, viewing the left side of \eqref{LinearForm} as a small linear form, with details  as  follows.\\

Let $\bbL := \bbQ(\sqrt{\alpha}, \sqrt{\beta})$. We have $D=[\bbL:\bbQ]\le 4$. We denote by $\pi$ and $ \opi$ be the prime ideals dividing $q$ in $\calO_\bbL$. 
 Let $\calS$ be the subset of $\calM_\bbL$ consist of the three valuations corresponding to $\lbrace \pi,  \opi, \infty \rbrace $.  
We define the linear forms $L_{i,v}$ for $v\in \calS$ and $i=1, \dotso, 4$ as follows:
\begin{align*}
L_{1,\infty} &:=  X_1 -  X_2  + \frac{1}{2 } X_3 + \frac{1}{2 }  X_4,  \\
L_{i,\infty} & :=  X_i \qquad \mbox{for $i=2, \dotso, 4$}, 
\end{align*}
whereas for $v\in \calS, v\neq \infty$, we put 
\begin{align*}
L_{i,\infty} & :=  X_i \qquad \mbox{for $i=1, \dotso, 4$}. 
\end{align*}
Then $\lbrace L_{1,v}, \dotso, L_{4,v} \rbrace$, $v\in \calS$ are linearly independent sets of linear forms in 4 variables with coefficients in $\bbL$. 
Finally, we define the vector
\begin{align*}
\mathbf{x} = (X_1, \dotso, X_4) = ( uq^{n/2},  q^{n},  \alpha^n,  \beta^n  ) \in \bbL^4.
\end{align*}

With these settings, we can evaluate $H(L_{i,\mu})$ as follows.
For  all $\mu \in \calS$ and $i = 1, \dotso, 4$ except when $\mu = \infty, i=1$, we have $\mathbf{a_i} = (a_1, \dotso, a_4)$ with $a_i=1$ and other entries 0. Thus,
 \begin{align*}
H( L_{i,\mu} )  = H(\mathbf{a_i} ) 
= \prod_{\mu \in \calM_\bbL} \left\vert \mathbf{a_i}  \right\vert_\mu
= 1.
\end{align*}
For $\mu=\infty, i=1$, we have $\mathbf{a} = ( 1, -1,  \frac{1}{2}, \frac{1}{2}  ) $ so that
 \begin{align*}
H( L_{1, \infty} ) & = H  ( \mathbf{a} ) 
= \prod_{\mu \in \calM_\bbL} \left\vert \left( 1, -1, \frac{1}{2}, \frac{1}{2}  \right) \right\vert_\mu  
 \leq 2 = \widetilde{H}.
\end{align*}
Let us note that $H(L_{i,\mu}) \leq \max \lbrace  1, \widetilde{H} \rbrace = \widetilde{H}$. \\

We need to find some explicit $\delta$ with $0<\delta<1$ such that the inequality \eqref{Subspace_ToSatisfy} is satisfied.  
First, we consider the expression 
$ \prod_{\mu \in \calS}	\left\vert \det\left( L_{1,\mu}, \dotso, L_{4,\mu} \right) \right\vert_\mu $ 
and observe that for $\mu=\infty$ we have
$
\left\vert \det\left( L_{1,\infty}, \dotso, L_{4,\infty} \right) \right\vert_\infty
=1
$
as the corresponding matrix is upper-triangular,
whereas 
$ 
\left\vert \det\left( L_{1,\mu}, \dotso, L_{4,\mu} \right) \right\vert_\mu
  = \left \vert 
\det I   
\right \vert_\mu
=1 $
for $\mu \in \calS$, $\mu \neq \infty$, where $I$ is the identity matrix.  
Next, we rewrite
\begin{align*}
\prod_{\mu \in \calS} \prod_{i=1}^4 \left\vert L_{i,\mu}(\mathbf{x}) \right\vert_\mu
& =\left(\prod_{\mu \in \calS}   \left\vert L_{1,\mu}(\mathbf{x}) \right\vert_\mu \right) \left(\prod_{\mu \in \calS} \prod_{i=2}^4 \left\vert L_{i,\mu}(\mathbf{x}) \right\vert_\mu \right).
\end{align*}
We apply the product formula \eqref{prod1} to get 
\begin{align*}
\prod_{\mu \in \calS} \prod_{i=2}^4 \left\vert L_{i,\mu}(\mathbf{x}) \right\vert_\mu = \prod_{i=2}^4 \dfrac{1}{\prod_{\mu \notin \calS} \left\vert X_i \right\vert_\mu}  =1. 
\end{align*}
Besides, using the fact that $u\in \bbZ$ so that $ \left\vert u \right\vert_\pi \leq 1 $, $\left\vert u \right\vert_{\opi} \leq 1$ and 
$ \left\vert L_{1,\infty}(\mathbf{x}) \right\vert_\infty <4.6$ 
as in \eqref{LinearForm},   we have
\begin{align*}
\prod_{\mu \in \calS} \prod_{i=1}^4 \left\vert L_{i,\mu}(\mathbf{x}) \right\vert_\mu
&  =  \left\vert X_1 \right\vert_\pi 
    \left\vert X_1 \right\vert_{\opi}
      \left\vert L_{1,\infty}(\mathbf{x}) \right\vert_\infty 
  =  \left\vert u q^{\frac{n}{2}} \right\vert_\pi  \,
    \left\vert  u q^{\frac{n}{2}} \right\vert_{\opi} \,
  \left\vert L_{1,\infty}(\mathbf{x}) \right\vert_\infty  \\
& <4.6 \left\vert q^{\frac{n}{2}}  \right\vert_\pi \,
    \left\vert   q^{\frac{n}{2}}  \right\vert_{\opi} 
= 4.6 q^{-\frac{n}{2}}. 
\end{align*}

Now, we note that $ H(\mathbf{x}) = \prod_{\mu \in \calM_\bbL} 
 {\vert \bfx \vert_\mu}
 = \left( \prod_{\mu \in \calS} 
 {\vert \bfx \vert_\mu} \right) \left( \prod_{\mu \notin \calS}  {\vert \bfx \vert_\mu} \right) $, 
 where 
$$
 \prod_{\mu \notin \calS}  {\vert \bfx \vert_\mu} 
  =  \prod_{\mu \notin \calS}  \max\left(|X_1|_\mu, \dotso, |X_4|_\mu \right) = 1.
  $$ 
  Therefore, we get
\begin{align*}
\prod_{\mu \in \calS} \prod_{i=1}^4 
\dfrac{1}{\vert \bfx \vert_\mu}
= \left( \prod_{\mu \in \calS}  
 \vert \bfx \vert_\mu  \right)^{-4}
= \left(\left( \prod_{\mu \in \calS}  
 \vert \bfx \vert_\mu  \right)\left( \prod_{\mu \notin \calS}  
 \vert \bfx \vert_\mu  \right)\right)^{-4}
 = \left(H(\bfx)\right)^{-4}.
\end{align*}
 
Besides, we deduce  that
\begin{align*}
{\vert \bfx \vert_\pi} 
& =  \max\left(|X_1|_\pi, \dotso, |X_4|_\pi \right)  
  =   \max\left(| u q^{n/2} |_\pi,  |q^{n }|_\pi,  | \alpha^n|_\pi, | \beta^n|_\pi  \right)  
 =1, \\
{\vert \bfx \vert_{\opi}}
 & =  \max\left(|X_1|_{\opi}, \dotso, |X_4|_{\opi} \right) 
 =   \max\left(|u q^{n/2} |_{\opi},  |q^{n }|_{\opi},    | \alpha^n|_{\opi}, | \beta^n|_{\opi}   \right)  
 =1. 
\end{align*}
Next, we bound ${\vert \bfx \vert_\infty}$. Using \eqref{u_def} and Lemma \ref{Lemma_O(k+1)}, we rewrite $\vert X_1 \vert_\infty$ as
\begin{align*}
\vert X_1 \vert_\infty 
& = \vert u q^{n/2} \vert_\infty 
= \left\vert  q^{n } \left( 1 + \sum_{r=1}^\infty \binom{1/2}{r} w^r \right) \right\vert_\infty  
< 1.0003 q^n,
\end{align*}
so that with $q \geq 2$ and the assumption $n\geq 30$,  
\begin{align*}
{\vert \bfx \vert_\infty}
& = \left( \sum_{i=1}^4 \vert X_i \vert_\infty^D \right)^{1/D}  
 = \left( \vert X_1 \vert_\infty^4  +  \vert q^{n } \vert_\infty^4 +  \vert \alpha^n \vert_\infty^4 +  \vert \beta^n \vert_\infty^4   \right)^{1/4} \\
& < \left( q^{4n} \left( 1.0003^4 +  1 +  \frac{2}{q^{2n}}   \right) \right)^{1/4} 
\leq \left( q^{4n} \left( 1.0003^4 +  1 +  \frac{2}{2^{60}}   \right) \right)^{1/4} \\
& < 1.2 q^{n}.
\end{align*}
This yields 
\begin{align*}
H(\bfx) =   \prod_{\mu \in \calM_\bbL} 
 {\vert \bfx \vert_\mu}
 = \left( \prod_{\mu \in \calS} 
 {\vert \bfx \vert_\mu} \right) \left( \prod_{\mu \notin \calS}  {\vert \bfx \vert_\mu} \right) 
 =  {\vert \bfx \vert_\pi} \, {\vert \bfx \vert_{\opi}}  \, {\vert \bfx \vert_\infty} 
< 1.2 q^{n},
\end{align*}
 and hence
$H(\bfx)^{-\delta} 
> 0.8 q^{-n\delta}$. 
 We note that for $q\geq 2$ and $n \geq 30$,
\begin{align*}
\prod_{\mu \in \calS} \prod_{i=1}^4 \left\vert L_{i,\mu}(\mathbf{x}) \right\vert_\mu
& < 4.6 q^{-\frac{n}{2}} 
= \left( 4.6 q^{-\frac{1}{10}n} \right) q^{-\frac{2}{5}n} 
\leq \left( \frac{4.6}{2^{3}}   \right) q^{-\frac{2}{5}n} 
< 0.8 q^{-\frac{2}{5}n}. 
\end{align*}

Finally, we take $\delta:=2/5$, and then the above estimates altogether give
\begin{align*}
\prod_{\mu \in \calS} \prod_{i=1}^4 \frac{|L_{i,\mu}(\bfx)|_{\mu}}{|\bfx|_\mu} 
<0.8 q^{-\frac{2}{5}n} H(\bfx)^{-4}
\leq H(\bfx)^{-4-\delta} \prod_{\mu \in \calS} |\det(L_{1,\mu}, \dotso, L_{4,\mu})|_\mu,
\end{align*}
and thus \eqref{Subspace_ToSatisfy} is satisfied.
Now we apply Lemma \ref{Lemma_Subspace} with $m=4$, $D=4$, $\# \calS=3$, to deduce that there exist $t$ proper linear subspaces $\calT_1, \dotso, \calT_t$ of $\bbL^4$, with
\begin{equation*}
\begin{split}
t & \leq (3m)^{2m \# \calS} 2^{3(m+9)^2} \delta ^{-m \# \calS - m - 4} \log 4D \log \log 4D  
< 10^{187}
\end{split}
\end{equation*}
such that every solution $\bfx \in \calX$ with
\begin{equation} \label{H_low2}
H(\bfx) \geq \max \lbrace m^{4m/\delta}, H \rbrace
\geq 1.2\times 10^{24}
\end{equation}
belongs to $\calT_1 \cup \cdots \cup \calT_t$. \\

Since we are in case (a), we have that ${\alpha}/{\beta}$ is not a root of unity. 
Let $\calT$ be one of these subspaces and suppose that $\calT$ does not depend on $X_1$. 
Then the solution 
$$
\bfx = ( uq^{n/2}, q^{n}, \alpha^n,  \beta^n) \in \calX
$$ 
satisfying \eqref{H_low2} satisfies an equation of the form
\begin{equation*}
 a_1  q^{n} + a_2 \alpha^n + a_3 \beta^n =0
\end{equation*}
for some vector of coefficients 
$(a_1, a_2, a_3) \in \bbL^3$ not all zero.
By Lemma \ref{Lemma_vdPSch}, this relation can hold for at most 
\begin{align*}
c(K, D, \omega) = c(3, 4, 2) = (3-1)(4(4+2))^{2(2)+1} = 15925248
\end{align*}
values of $n$. 
Suppose next that $\calT$ is one of these subspaces and  $\calT$   depends on $X_1$.  
Then the solution $\bfx = ( uq^{n/2}, q^{n}, \alpha^n,  \beta^n) \in \calX$ with \eqref{H_low2} satisfies an equation of the form
\begin{equation*}  
 uq^{n/2}+ a_1  q^{n} + a_2 \alpha^n + a_3  \beta^n =0
\end{equation*}
for some vector of coefficients 
$(1, a_1, a_2, a_3) \in \bbL^4$ not all zero. Together with \eqref{u2}, we have
\begin{align*} 
u^2 q^n = q^{2n} + q^n  - q^n \alpha^n - q^n \beta^n =  (a_1  q^{n} + a_2 \alpha^n + a_3  \beta^n)^2,
\end{align*}
yielding
\begin{align*} 
c_1 q^{2n} + c_2 \alpha^{2n} + c_3 \beta^{2n}  + c_4 q^n \alpha^n  + c_5 q^n \beta^n + c_6 q^n = 0,
\end{align*}
where $c_j \in \bbL, \; 1\leq j \leq 6$.
By Lemma \ref{Lemma_vdPSch}, this can hold for at most 
\begin{align*}
c(K, D, \omega) = c(6, 4, 2) = (6-1)(4(4+2))^{2(2)+1} = 39813120
\end{align*}
values of $n$.

Next, we consider the solutions of ``small height" $\bfx = ( uq^{n/2}, q^{n}, \alpha^n,  \beta^n) \in \calX$ with 
\begin{equation*} 
H(\bfx) < \max \lbrace m^{4m/\delta}, H \rbrace
< 1.3\times 10^{24}.
\end{equation*}
Note that
\begin{equation} \label{Hxlow}
\begin{split}
H(\bfx)
& = \left( \prod_{\mu \in \calS}  
 \vert \bfx \vert_\mu \right)\left( \prod_{\mu \notin \calS}  
 \vert \bfx \vert_\mu \right)  
= \prod_{\mu \in \calS}  
 \vert \bfx \vert_\mu  
 =  {\vert \bfx \vert_\pi} \, {\vert \bfx \vert_{\opi}}  \, {\vert \bfx \vert_\infty} 
= \left( \sum_{i=1}^4 \vert X_i \vert_\infty^D \right)^{1/D}   \\
&  > \vert X_1 \vert_\infty 
= \vert u q^{n/2} \vert_\infty 
= \left\vert  q^{n } \left( 1 + \sum_{r=1}^\infty \binom{1/2}{r} w^r \right) \right\vert_\infty \\ 
& > q^n.
\end{split}
\end{equation}
Thus, we have 
$$
n  < \frac{\log(1.3 \times 10^{24})}{\log q} < \frac{\log(1.3 \times 10^{24})}{\log 2}< 80.
$$
Therefore, we obtained the upper bound 
\begin{align*}
 80+ (15925248+39813120)10^{187} < 5.6 \times 10^{194}
\end{align*}
on the number of possible values of $n$ such that $\#E({\mathbb F}_{q^n})$ is a square, which finishes the proof of part (a). 


\section{Proof of Theorem \ref{Thm_Main}(b) and (c)}  \label{Sec:ProofThmMain}

In this section, we assume that ${\alpha}/{\beta} = \zeta = e^{\frac{2\pi ki}{m}}$ is  a root of unity with $(k, m)=1$. Since $\alpha$, $\beta$ are roots of $x^2-ax+q=0$, we have $\deg (\zeta  ) = \deg( e^{\frac{2\pi ki}{m}} ) = \phi(m) \leq 2$, giving $m=1$, $2$, $3$, $4$ or $6$. This takes care of the first assertion from (b). 
We shall show that  $ \# E(\bbFqn)$ is a perfect square if $n \equiv 0 \;\mathrm{mod}\;m$. For $n \not\equiv 0 \;\mathrm{mod}\;m$, we find all $n$ such that $\# E(\bbFqn)$ is a perfect square. 
We consider each value of $m$ below.\\

\noindent
\textbf{(1). $m=1$}\\
This gives ${\alpha}/{\beta} = \zeta = e^{ 2\pi ki } = 1$, that is $\alpha = \beta \in \bbZ$ and hence 
$\alpha = \beta = \pm q^{1/2}$.
This yields
\begin{align*}
a = \pm 2\sqrt{q}.
\end{align*}
In order that $a\in \bbZ$, we need $\sqrt{q}\in \bbZ$, therefore
\begin{align*}
q = p^{2v} \mbox{ and } a &=\pm 2p^{v}, \quad \mbox{where $v \in \bbZ$.}
\end{align*}
Then for all $n\in \bbZ$, $a_n =  \alpha^n + \beta^n = 2(\pm q^{1/2})^n$ and
\begin{align*}
\# E(\bbFqn) 
&= q^n+1-\alpha^n- \beta^n  
  = q^n+1 - 2(\pm q^{1/2})^n 
  = ((\pm q^{1/2})^n - 1)^2
  = ((\pm p^{v})^n - 1)^2
\end{align*}
is a perfect square.\\

\noindent
\textbf{(2). $m=2$}\\
This gives $k=1$, ${\alpha}/{\beta} = \zeta = e^{  \pi  i } = -1$, that is $\alpha = -\beta \in \bbZ$ and hence without loss of generality we may assume that
$\alpha =  q^{1/2} i$  and $\beta = -q^{1/2} i$.
This yields to $a=0$.
When $n$ is odd,  
\begin{align*}
\# E(\bbFqn) 
&= q^n+1-\alpha^n- \beta^n  
  = q^n+1.  
\end{align*}
By known results about the Catalan equation, the only solution when either $n>1$ or when $n=1$ and $q$ is not prime is $2^3+1=3^2$. This gives the solutions $(q,a,n)=(2,0,3),~(8,0,1)$. When 
$n=1$ and $q=p$ is prime, then setting $p+1=u^2$, we get $p=(u-1)(u+1)$. Since $p$ is prime, the only possibility is $u-1=1$, so $(u,p)=(2,3)$, which gives the solution $(q,a,n)=(3,0,1)$. 

When $n \equiv 2 \mbox{ mod }4$,
$a_n = a_{2r} =  \alpha^n + \beta^n = - 2q^{n/2}$ and
\begin{align*}
\# E(\bbFqn) 
&= q^n+1-\alpha^n- \beta^n  
  = q^n+1 + 2q^{n/2} 
  = (q^{n/2} + 1)^2
\end{align*}
is a perfect square for all $q$.\

When $n \equiv 0 \mbox{ mod }4$,
$a_n = a_{2r} =  \alpha^n + \beta^n =  2q^{n/2}$ and
\begin{align*}
\# E(\bbFqn) 
&= q^n+1-\alpha^n- \beta^n  
  = q^n+1 - 2q^{n/2} 
  = (q^{n/2} - 1)^2
\end{align*}
is a perfect square for all $q$. \

Finally, we note from Lemma \ref{Lemma_Waterhouse} that there is no corresponding elliptic curve for $(q, a)$ if $a=0$, $q=p^{2v}$ where  $p \equiv 1 \mbox{ mod }4, v\in \bbZ$. \\

\noindent
\textbf{(3). $m=3$}\\
This gives  ${\alpha}/{\beta} = \zeta = e^{\frac{2  \pi k  i}{3}  }$. Then either (i).
$\alpha =  q^{1/2} e^{\frac{ \pi  i}{3}  }$  and $\beta = q^{1/2} e^{-\frac{  \pi  i}{3}  }$ for $k=1$, or
(ii). 
$\alpha =  q^{1/2} e^{\frac{ 2 \pi  i}{3}  }$  and $\beta = q^{1/2} e^{-\frac{ 2 \pi  i}{3}  }$ for $k=2$. 
Thus, we have either \\
 (i). \quad $a =  q^{1/2}, \quad 
 \# E(\bbFqn) 
 = q^n+1-\alpha^n- \beta^n  
  = q^n+1 - 2 q^{n/2} \cos \left(\frac{n\pi}{3} \right)  $, 
or \\
 (ii).  \quad $a =  -q^{1/2}, \quad 
 \# E(\bbFqn) 
 = q^n+1-\alpha^n- \beta^n  
  = q^n+1 - 2 q^{n/2} \cos \left(\frac{2n\pi}{3} \right)$. \\
We shall consider different scenarios for $n$.\\

When $n \equiv 0 \mbox{ mod }6$, in both (i) and (ii) we have  
\begin{align*}
 \# E(\bbFqn) 
&=  q^n+1 - 2 q^{n/2}  = (  q^{n/2}  -1)^2. 
\end{align*}
This is a perfect square whenever $q^{n/2}  -1 \in \bbZ$; i.e.,
$q = p^{2v}, \quad v\in \bbZ$.
Therefore, the corresponding $a$ and $a_n$ are\\
(i). \quad $a=p^v, a_n= 2 p^{vn}  $, \qquad
(ii). \quad $a=-p^v, a_n= 2 p^{vn}  $.\\

When $n \equiv 3 \mbox{ mod }6$, we have  \\
(i). \quad $a =  q^{1/2}, \quad 
 \# E(\bbFqn) 
  = q^n+1 + 2 q^{n/2} = (q^{n/2}+1)^2 $, 
or \\
 (ii).  \quad $a =  -q^{1/2}, \quad 
 \# E(\bbFqn) 
  = q^n+1 -2 q^{n/2} = (q^{n/2}-1)^2$. \\
This is a perfect square whenever $q^{n/2} \pm 1 \in \bbZ$; i.e.,
$q = p^{2v}, \quad v\in \bbZ$.
Therefore, the corresponding $a$ and $a_n$ are\\
(i). \quad $a=p^v, a_n= 2 p^{vn}  $, \qquad
(ii). \quad $a=-p^v, a_n= -2 p^{vn}  $.\\
Finally, we note from Lemma \ref{Lemma_Waterhouse} that there is no corresponding elliptic curve for $(q, a)$ if $a=\pm \sqrt{q}$, $q=p^{2v}$ with $p \equiv 1 \mbox{ mod }3$. \\

This was for $3\mid n$. If $3\nmid n$, then both $\cos(n\pi/3),\cos(2n\pi /3)\in \{\pm 1/2\}$. So, we are lead to the equations $q^n\pm q^{n/2}+1=u^2$. We may assume that $u$ is a positive integer. 
This gives $(u-q^{n/2})(u+q^{n/2})=\pm q^{n/2}+1$. 
In case the sign is $+$ in the right--hand side above, we then get that $u-q^{n/2}>0$. Thus, the number $u+q^{n/2}>2q^{n/2}$ is a factor of $q^{n/2}+1<2q^{n/2}$, which is impossible. In case 
the sign is $-$ in the right--hand side above, we get that $u+q^{n/2}>q^{n/2}$ is a factor of $q^{n/2}-1<q^{n/2}$, which is again impossible. So, there are no solutions with $n$ coprime to $3$ in this case. \\

\noindent
\textbf{(4). $m=4$}\\
This gives $\frac{\alpha}{\beta} = \zeta = e^{\frac{2 \pi k  i}{4}  } $. Then either (i).
$\alpha =  q^{1/2} e^{\frac{ \pi  i}{4}  }$  and $\beta = q^{1/2} e^{-\frac{  \pi  i}{4}  }$ for $k=1$, or
(ii). 
$\alpha =  q^{1/2} e^{\frac{ 3 \pi  i}{4}  }$  and $\beta = q^{1/2} e^{-\frac{ 3 \pi  i}{4}  }$ for $k=3$. 
Thus we have either \\
 (i). \quad $a =  \sqrt{2 q }, \quad 
 \# E(\bbFqn) 
 = q^n+1-\alpha^n- \beta^n  
  = q^n+1 - 2 q^{n/2} \cos \left(\frac{n\pi}{4} \right)  $, 
or \\
 (ii).  \quad $a =  -\sqrt{2 q }, \quad 
 \# E(\bbFqn) 
 = q^n+1-\alpha^n- \beta^n  
  = q^n+1 - 2 q^{n/2} \cos \left(\frac{3n\pi}{4} \right)$. \\
We shall consider different scenarios for $n$.\\

When $n \equiv 0 \mbox{ mod }8$, in both (i) and (ii) we have  
\begin{align*}
 \# E(\bbFqn) 
&=  q^n+1 - 2 q^{n/2}  = (  q^{n/2}  - 1)^2 ,
\end{align*}
which is a perfect square for all $q$.
In order that $a \in \bbZ$, we need $ \sqrt{2q} \in \bbZ$, implying $q=2^v$ with odd $v$.
Therefore, the corresponding   $a$ and $a_n$ are\\
(i). \quad  $a=2^{\frac{v+1}{2}}, a_n= -2^{\frac{nv+2}{2}}  $, \qquad
(ii). \quad $a=-2^{\frac{v+1}{2}}, a_n= -2^{\frac{nv+2}{2}}$.\\

When $n \equiv 4 \mbox{ mod }8$, in both (i) and (ii) we have  
\begin{align*}
 \# E(\bbFqn) 
&=  q^n+1 + 2 q^{n/2}  = (  q^{n/2}  + 1)^2,
\end{align*}
which is a perfect square for all $q$. 
In order that $a \in \bbZ$, we need $ \sqrt{2q} \in \bbZ$, implying $q=2^v$ with odd $v$.
Therefore, the corresponding   $a$ and $a_n$ are\\
(i). \quad  $a=2^{\frac{v+1}{2}}, a_n= -2^{\frac{nv+2}{2}}  $, \qquad
(ii). \quad $a=-2^{\frac{v+1}{2}}, a_n= -2^{\frac{nv+2}{2}}$.\\
 
Assume next that $n\not\equiv 0\mbox{ mod } 4$. If $n$ is even, then $n\equiv 2,6 \mbox{ mod } 8$ so $a_n=0$. This leads to $u^2=q^n+1$, a case which has been dealt with at the case $m=2$ above. 
If $n$ is odd, then   $\cos(n\pi/4),\cos(3n\pi/4)\in \{\pm 2^{-1/2}\}$, so we get $u^2=q^n\pm 2^{1/2} q^{n/2}+1$. Thus $2^{1/2} q^{n/2}\in {\mathbb Z}$, which shows that $q=2^b$ with $b$ and $n$ odd. Hence, 
$u^2=2^{bn}\pm 2^{(bn+1)/2}+1$. The equations $u^2=2^x\pm 2^y+1$ with $x\ge y$ have been solved by Szalay in \cite{Sz}. Aside from the parametric solutions with $(x,y)=(2t,t+1)$ for both signs and 
$(x,y)=(t,t)$ for the sign $-$, it has the sporadic solutions $(x,y)=(5,4),~(9,4)$ for the sign $+$ and $(x,y)=(5,3),~(7,3),~(15,3)$ for the sign $-$. Thus, we get that either $(bn,(bn+1)/2))=(2t,t+1),~(t,t)$ 
for some integer $t$, or it is one of the $5$ sporadic solutions. The possibility $(bn, (bn+1)/2)=(2t,t+1)$ is not convenient since for us both $b$ and $n$ are both odd. The solution $(bn,(bn+1)/2)=(t,t)$ leads to 
$bn=(bn+1)/2$, which gives $b=n=1$, so $(q,a,n)=(2,2,1)$. Of the remaining $5$ sporadic solutions only $(bn,(bn+1)/2)=(5,3)$ is convenient and leads to $bn=5$, so $(q,n)=(5,1),~(2,5)$. 
This leads to $(q,a,n)=(32,8,1),~(2,-2,5)$. \\


\noindent
\textbf{(5). $m=6$}\\
This gives ${\alpha}/{\beta} = \zeta = e^{\frac{ \pi k i}{3}  } $. Then either (i).
$\alpha =  q^{1/2} e^{\frac{\pi i}{6} }$  and $\beta = q^{1/2} e^{-\frac{  \pi  i}{6} }$ for $k=1$, or
(ii). 
$\alpha =  q^{1/2} e^{\frac{ 5 \pi  i}{6}  }$  and $\beta = q^{1/2} e^{-\frac{ 5\pi  i}{6}  }$ for $k=5$.  
Thus, we have either \\
 (i). \quad $a =  \sqrt{3 q }, \quad 
 \# E(\bbFqn) 
 = q^n+1-\alpha^n- \beta^n  
  = q^n+1 - 2 q^{n/2} \cos \left(\frac{n\pi}{6} \right)  $, 
or \\
 (ii).  \quad $a =  -\sqrt{3 q }, \quad 
 \# E(\bbFqn) 
 = q^n+1-\alpha^n- \beta^n  
  = q^n+1 - 2 q^{n/2} \cos \left(\frac{5n\pi}{6} \right)$. \\
We shall consider different scenarios for $n$.\\

When $n \equiv 0 \mbox{ mod }12$, in both (i) and (ii) we have  
\begin{align*}
 \# E(\bbFqn) 
&=  q^n+1 - 2 q^{n/2}  = (  q^{n/2}  - 1)^2,
\end{align*}
which is a perfect square for all $q$.
In order that $a \in \bbZ$, we need $ \sqrt{3q} \in \bbZ$, implying $q=3^v$ with odd $v$.
Therefore, the corresponding   $a$ and $a_n$ are\\
(i). \quad  $a=3^{\frac{v+1}{2}}, a_n= 2(3^{\frac{nv }{2}} ) $, \qquad
(ii). \quad $a=-3^{\frac{v+1}{2}}, a_n=2(3^{\frac{nv }{2}} )$.\\

When $n \equiv 6 \mbox{ mod }12$, in both (i) and (ii) we have  
\begin{align*}
 \# E(\bbFqn) 
&=  q^n+1 + 2 q^{n/2}  = (  q^{n/2}  + 1)^2, 
\end{align*}
which is a perfect square for all $q$. 
In order that $a \in \bbZ$, we need $ \sqrt{3q} \in \bbZ$, implying $q=3^v$ with odd $v$.
Therefore, the corresponding   $a$ and $a_n$ are\\
(i). \quad  $a=3^{\frac{v+1}{2}}, a_n= -2(3^{\frac{nv }{2}} ) $, \qquad
(ii). \quad $a=-3^{\frac{v+1}{2}}, a_n=-2(3^{\frac{nv }{2}} )$.\\

When $n$ is an odd multiple of $3$, we get that $  \# E(\bbFqn)   = q^n+1 $, a case treated at the case $m=2$. When $n$ is even and coprime to $3$, we get 
$u^2=q^{n}\pm q^{n/2}+1$, a case already treated at $m=3$ above. Finally, when $n$ is coprime to $6$, then 
$\cos(n\pi/6), \cos(5n\pi/6)\in \{\pm 3^{1/2}/2\}$. In this case, we get $u^2=q^n\pm 3^{1/2} q^{n/2}+1$. Since $3^{1/2}q^{n/2}\in {\mathbb Z}$, it follows that $q=3^b$ with $b$ and $n$ both odd. 
The equation $u^2=p^x\pm p^y+1$ with an odd prime $p$ and integers $x> y$ has been solved by Luca \cite{Luca}. Its only solutions are $(p,x,y)=(3,3,1),~(5,3,1)$. Thus, if 
$bn>1$, then the only possibility is $(bn,(bn+1)/2)=(3,1)$, which does not have a convenient integer solution $b,n$. The solution with $bn=1$ leads to $b=n=1$, so $(q,a,n)=(3,3,1)$.  This finishes the proof of our theorem.\\

\textbf{Acknowledgements.}   K.~C.~Chim is supported by the Austrian Science Fund (FWF) under the project F5510-N26. Most of the work took place when both authors were visiting the Max Planck Institute for Mathematics
in Bonn, Germany between October and November of 2019. The authors would like to thank the institute for the hospitality, nice working environment and computer support.

\bibliographystyle{plain}

\bibliography{Rec}

\end{document}